\documentclass[12pt, reqno]{amsart}
\usepackage{latexsym,amsmath,amsopn,amssymb,amsthm,amsfonts}
\usepackage[T2A]{fontenc}
\usepackage[cp1251]{inputenc}
\usepackage{graphicx}

\newtheorem{theorem}{Theorem}

\newtheorem{proposition}{Proposition}

\newtheorem{remark}{Remark}

\newtheorem{definition}{Definition}

\newtheorem{corollary}{Corollary}

\DeclareMathOperator{\ad}{ad}

\DeclareMathOperator{\Ad}{Ad}

\begin{document}

\title[The G\"odel Universe as the Lie group]{The G\"odel Universe as  the Lie Group with left-invariant Lorentz metric}
\author{V.~N.~Berestovskii}
\thanks{The work was carried out within the framework of the State Contract to the IM SB RAS, project FWNF-2022-0006.}
\address{Sobolev Institute of Mathematics of the SB RAS,\newline
4 Acad. Koptyug Ave., Novosibirsk 630090, Russia}
\email{vberestov@inbox.ru}

\begin{abstract}
The author studies the G\"odel Universe as the Lie group with left-invariant Lorentz metric. The expressions for timelike and isotropic geodesics in elementary functions are found by methods of geometric theory of optimal control for the search of geodesics on Lie groups with left-invariant (sub-)Lorentz metrics. It is proved that the G\"odel Universe has no closed
timelike or isotropic geodesics. 

\medskip

	
\vspace{2mm}
\end{abstract}    	
\maketitle

\section{Introduction}

Kurt G\"odel in paper \cite{God1949} of 1949 introduced in the space  $\mathbb{R}^4$ a Lorentz metric (\ref{lig}) with signature $(+,-,-,-).$ The G\"odel Universe (space-time) $S$ is a solution of Einstein's field equations of gravitation. 

In the first section of the paper, G\"odel indicates the main properties of
his (called new by him) solution. Let us mention some of them in other terms. 

All cosmological solutions with positive density of the matter, known before that time, admitted the {\it absolute time}, whose level hyperspaces are transversal to world lines of matter. (Let us note that usually such hyperspaces are called {\it Cauchy surfaces}). The nonexistence of such hyperspaces (by G\"odel's mind) is equivalent to the matter rotation relative to the inertia axes. G\"odel suggested the solution, the space-time $S,$ with negative cosmological term, with such rotation.   

G\"odel states that the space-time $S$ has the following properties:

1) The space-time  $S$ is homogeneous. Then $S$ is stationary.

3) $S$ has a rotation symmetry: for every point $P$ in $S$
there exists a one-parameter isometry group of $S$ with fixed point $P$.

4) The space of timelike and isotropic vectors possesses a temporal orientation. That is, it is possible to introduce consistently (continuously) a positive direction of the time into all system.

After that, it is defined the temporal ordering for any two neighboring points. On the other hand, however, there is no uniform temporal ordering 
of {\it all} points, agreeing with local orderings. This is expressed in the next property:

5) It is impossible to assign a time coordinate $t\in\mathbb{R}$ to each 
point of the space in such a way that $t$ always increases, if one moves in
the positive timelike direction.

6) Every world line of matter, occuring in the solution, is an open line of infinite length, which never approaches any of its preceding points again;
but there also exist closed timelike lines.

7) There exist no hyperspace in $S$ which is everywhere spacelike and intersects each world line of matter in exactly one point.

9) Matter everywhere rotates relative to the inertia axis with angular velocity $2(\pi\kappa\rho)^{1/2},$ where $\rho$ is the mean density of matter and $\kappa$ gravitational constant.

Interesting property 6) is described more exactly later in \cite{God1949} relative to closed timelike curves. An elegant and simple geometric interpretation of the ''world lines of matter'' (together with $S$) is given on pp. 449 and 450 in \cite{God1949}. 

World lines of matter are the simplest timelike geodesic (\ref{wl}) and all its left shifts (\ref{gr}) by elements of the Lie group $G$. These are exactly the inertia axes. Any of them is the axis of some one-parameter rotation group.

In \cite{God1949} is given also an explicit description of closed isotropic curves.

Notice that G\"odel says in the footnote to mention on the absolute time
about philosophical corollaries of this concept, and besides other sourses, he references to his essay \cite{God1949a} (which was also published in 1949) from the collection, dedicated to the 70th anniversary of Einstein, . 

Let us mention only the following words of Einstein in \cite{Einst} on  G\"odel's essay:

''Kurt G\"odel's essay constitutes, in my opinion, an important contribution
to the general theory of relativity, especially to the analysis of the concept of time. The problem here involved disturbed me already at the time of the building up of the general theory of relativity. ...
It will be interesting to weigh whether these (cosmological solution by G\"odel of Einstein's field equations with negative constant $\Lambda$) are not to be excluded on  physical grounds''.

In paper \cite{God1950}, mainly without proofs, G\"odel investigates the properties of families of rotating Universes (solutions) without closed
timelike and isotropic curves, with the null density of matter, of two types: I) spatially homogeneous and expanding; II) stationary and homogeneous.  

I) In this case 1) the spaces are compact, 2) the density is not constant.   
The last contition implies that the models are expanding.

II) In this case G\"odel formulates the following results.

1) There exist no stationary homogeneous solutions with $\Lambda=0.$

2) There exist rotation stationary homogeneous solutions with compact space, the absence of closed timelike curves, and $\Lambda>0,$
between them are arbitrary little differing from Einstein Universe.

These results of G\"odel precedes many investigation of such models. The work \cite{OzsSchuck} and other articles of its authors follow the G\"odel program.

In 1951, G\"odel received the Albert Einstein Award from the fund of Lewis and Rose Strauss for works on rotating Universes; Einstein himself handed in it.

Among the papers of G\"odel was found manuscript of his lecture ''On rotating Universes'' published in 1995 in the third volume of his Collected Works  \cite{GodColl}. G\"odel gave this lecture on May 7, 1949 at the Institute of Advanced Studies, Princeton. 

The literal quotation of the very first sentence of this lecture: 

''A few years ago, in a note in {\it Nature}, Gamov \cite{Gam1946} suggested that the whole universe might be in a state of uniform rotation and that this
rotation might explain the observed rotation of the galactic systems''.

The title of the next paper by Gamov \cite{Gam1952} suggests that the exact 
epithet for the G\"odel Universe is ''Turbulent G\"odel Universe''. 

The form of geodesics, discovered in theorems 1 and 2 of our article, and the homogeneity of G\"odel Universe supports this. 

In book \cite{ZeldNov} there is the section 15.9 with the synonymous title  ''The Vortex Theory''. It begins with statements that ''The vortex theory derives from the assumption that during an easrly RD stage the plasma is in a state of turbulent motion. In a qualitative form, the vortex theory has been propounded by Gamov (1952, 1954).'' 

The issues \cite{Kundt} --- \cite{Andr21}, disposed in the chronological order, have a direct relation to investigations of G\"odel Universe.
We shall say additionally about \cite{Kundt}, \cite{HockEll}, \cite{Lev19900} --- \cite{Steph}, \cite{Andr17}, \cite{Andr21} in section
\ref{last} of the paper.                
  
In paper \cite{ChWr}, S.~Chandrasekhar and J.~Wright found, as they say,
a full solution to equations of geodesics in $S$ (with application of Christoffel symbols).  

In this paper, we investigate G\"odel Universe $S$ as the Lie group $G$ 
with left-invariant Lorentz metric and find all timelike and isotropic 
geodesics in $S$. Their view implies that there are no closed geodesics among them.  We apply a method of geometric theory of the optimal control for the search of geodesics on Lie groups with left-invariant  (sub-)Lorentz metric. Under this geodesics are solutions to a system of ordinary differential equations of the first order while the corresponding system 
of equations for the classical method has the second order.

The applied method for the case of timelike geodesics was formulated in \cite{BerGich2001}. The first time this method in its partially linearized version was applied to solve a problem posed in \cite{Ber2023}. The results \cite{BerZub2} will be published in an issue of the journal ''Pure and Applied Functional Analysis'' dedicated to professor Anatoly Vershik in connection of his 90th anniversary (28.12.2023). 

Anatoly Moiseevich suddenly passed away on February 14, 2024. 

As a stimulus to write the paper \cite{Ber2023} served an interest to the Segal chronometric theory initiated by discussion with A.V.~Levichev.
The author is very grateful to A.V.~Levichev not only for this, but also for discussions on G\"odel Universe, among them for the indication to works \cite{Lev19900}--\cite{Lev1992} and exhibited interest.    

The author thanks I.A.~Zubareva for critique and great assistance in work.

\section{Timelike and isotropic geodesics}

In \cite{God1949}, G\"odel introduces his space-time $S$ as
$\mathbb{R}^4$ with linear element
\begin{equation}
\label{lig}
ds^2=a^2\left(dx_0^2+2e^{x_1}dx_0dx_2+ \frac{e^{2x_1}}{2}dx_2^2-dx_1^2-dx_3^2\right),\quad  a>0. 
\end{equation}
G\"odel notes in \cite{God1949} that this quadratic form can be written as 
\begin{equation}
\label{lig1}
ds^2=a^2\left[\left(dx_0+e^{x_1}dx_2\right)^2-dx^2_1-
\frac{e^{2x_1}}{2}dx^2_2-dx^2_3\right],	  
\end{equation}
which evidently implies that its signature is $(+,-,-,-)$ everywhere. 

G\"odel Universe is a solution to Einstein gravitational equations 
$$R_{ik}-\frac{1}{2}Rg_{ik}=8\pi\kappa\rho u_iu_k+\Lambda g_{ik},\quad T_{ik}=\rho u_iu_k;$$
$$R=\frac{1}{a^2},\quad R_{ik}=\frac{1}{a^2}u_iu_k,\quad \frac{1}{a^2}=8\pi\kappa\rho,\quad \Lambda= -\frac{R}{2},$$
where $\rho$ is a positive constant density of matter, $u$ is a normed 
vector of 4-velocity. ''The sign of the cosmological constant here is the opposite of that occuring in Einstein static solution. It corresponds to a positive pressure'' \cite{God1949}.

\begin{remark}
It is indicated in \cite{Kundt} and \cite{HockEll} that the pressure $p$ in $S$ is absent. Ellis say in \cite{Ellis} that there are many methods to construct G\"odel solution because of its large symmetry. If $\omega$ is the {\it vorticity}, then there are possible the cases: 
$$p=0,\quad \Lambda=-\omega^2<0,\quad (2.9);\quad \Lambda=0,\quad p=\omega^2/\kappa,\quad (2.10).$$	
\end{remark}	

We are interested mainly in geodesics. Timelike geodesics are locally longest timelike curves joining given pairs of points. Therefore we can assume that $a=1.$

Obviously, $(S,ds^2)=(S_0,ds_0^2)\times (S_1,ds_1^2),$ where $S_0=\mathbb{R}^3,$
$S_1=\mathbb{R},$
\begin{equation}
\label{lin}	
ds_0^2=dx_0^2+2e^{x_1}dx_0dx_2+ \frac{e^{2x_1}}{2}dx_2^2-dx_1^2,\quad ds_{1}^2=-dx_3^2.	
\end{equation}
This implies that the search for geodesics in $(S,ds^2)$ reduces to the search in $(S_0,ds_0^2)$.

Further G\"odel introduces in $S_0$ ''cylindrical coordinates'' $(r,\varphi,t)$ by formulae
$$e^{x_1}=\cosh 2r + \cos\varphi\sinh 2r,$$
$$x_2e^{x_1}=\sqrt{2}\sin\varphi\sinh 2r,$$
$$\tan\left(\frac{\varphi}{2}+\frac{x_0-2t}{2\sqrt{2}}\right)=
e^{-2r}\tan\left(\frac{\varphi}{2}\right),
\quad\mbox{where}\quad \left|\frac{x_0-2t}{2\sqrt{2}}\right|< \frac{\pi}{2}.$$
This leads to the following form of the linear element:
\begin{equation}
\label{cyl} 
ds_0^2=4[dt^2-dr^2+\sinh^2r((\sinh^2r-1)d\varphi^2+2\sqrt{2}d\varphi dt)].
\end{equation}
Here G\"odel omits computations, noting that they are cumbersome. The components of this linear element do not depend on $t$ and $\varphi$. Then, in these coordinates the shifts by $t$ and rotations around axis $r=0$ by $\varphi$ are isometries in $(S_0,ds_0^2).$ 

Taking account of G\"odel's remark in \cite{God1949}, that on $(S,ds^2)$ ($(S_0,ds_0^2)$) acts simply transitively four-dimensional (respectively, three-dimensional) isometry Lie group $G$ (respectively, $G_0$), we get that for any point $P\in S$ ($P_0\in S_0$) there exists a one-parameter group  of rotations--isometries fixing the point $P$ ($P_0$).

\begin{remark}
\label{cylin}	
It follows from the given above formulae for coordinates
$(r,\varphi,t)$ that $x_1=x_2=0$ and $x_0=2t$ (for close to zero $x_0, t$) if $r=0.$ This coordinate system has a singularity since for $r=0$ the Jacobian of transition from $(r,\varphi,t)$ to $(x_0,x_1,x_2)$ is equal to zero \cite{Lev1992}. Therefore later we do not apply in fact these coordinates. 
\end{remark}

It is easy to see that the mentioned action of the Lie group $G$ is given by formulae
\begin{equation}
\label{gr}
x_0=x_0'+a,\quad x_1=x_1'+b,\quad x_2=x_2'e^{-b}+c,\quad x_3=x_3'+d
\end{equation}
with arbitrary $a,b,c,d\in\mathbb{R}.$ This implies that the Lie group $G$
is a simplest noncommutative four-dimensional Lie group of the form
$$G\cong (\mathbb{R},+)\times G_{2}\times (\mathbb{R},+),$$ 
where $G_2$ is unique up to isomorphism, diffeomorphic to $\mathbb{R}^2$, two-dimensional noncommutative Lie group. In addition,
\begin{equation}
\label{g1}
G_0\cong (\mathbb{R},+)\times G_{2}.
\end{equation}

In case under consideration, identifying the triple $(x'_0,x_1',x_2')$ with the vector--column $(x_2',x_1',x'_0,1)^T,$ where $^T$ is the transposition sign,
the action of the group $G_0$ on $\mathbb{R}^3$ by formulae $(\ref{gr})$ has a form $(x_2,x_1,x_0,1)^T=A(x_2',x_1',x'_0,1)^T,$ where 
\begin{equation}
\label{A}
A= \left(\begin{array}{cccc}
	e^{-b} & 0 & 0 & c \\
	0 & 1 & 0 & b \\ 
	0 & 0 & 1 & a\\
    0 & 0 & 0  & 1\end{array}\right).
\end{equation}
Moreover, the equality 
\begin{equation}
\label{B}
\left(\begin{array}{cccc}
e^{-x_1} & 0 & 0 & x_2 \\
0 & 1 & 0 & x_1 \\ 
0 & 0 & 1 & x_0\\
0 & 0 & 0 & 1 \end{array}\right)(0,0,0,1)^T=(x_2,x_1,x_0,1)^T
\end{equation}
establishes a bijection of the group $G_0$ onto $\mathbb{R}^3$ and the unit of 
the group $G_0$ corresponds to zero-vector $(0,0,0)\in\mathbb{R}^3.$ On the base of this, (\ref{g1}) and (\ref{lin}),  we can identify $(S_0,ds_0^2)$
with the Lie group $G_0$ supplied by left-invariant Lorentz metric.	

As a corollary of the above and (\ref{lin}) in the unit of group $G_0,$ the components of linear element $ds_0^2$ relative to natural basis  $(e_0=e_{x_0},e_1=e_{x_1},e_2=e_{x_2}),$ connected with coordinates $(x_0,x_1,x_2),$ are equal to
\begin{equation}
\label{g} 
g_{00}=1,\quad g_{02}=g_{20}=1,\quad g_{22}=\frac{1}{2},\quad g_{11}=-1,\quad g_{01}=g_{10}=g_{12}=g_{21}=0.
\end{equation}
In a consequence of (\ref{B}), in the Lie algebra $\mathfrak{g}_0$ of the Lie group $G_0,$ 
\begin{equation}
\label{e}	
e_0 = \left(\begin{array}{cccc}
	0 & 0 & 0 & 0 \\
	0 & 0 & 0 & 0 \\
	0 & 0 & 0 & 1\\ 
	0 & 0 & 0 & 0\end{array}\right),\quad	
e_1 = \left(\begin{array}{cccc}
	-1 & 0 & 0 & 0 \\
	0 & 0 & 0 & 1 \\
	0 & 0 & 0 & 0\\ 
	0 & 0 & 0 & 0\end{array}\right),\quad 
e_2 = \left(\begin{array}{cccc}
	0 & 0 & 0 & 1 \\
	0 & 0 & 0 & 0\\
	0 & 0 & 0 & 0 \\ 
	0 & 0 & 0 & 0\end{array}\right).	 
\end{equation}
As a corollary, for the Lie algebra $\mathfrak{g}_0$,
\begin{equation}
\label{sc}
[e_1,e_2]=e_1e_2-e_2e_1=-e_2, \quad [e_0,e_1]=[e_0,e_2]=0.	 
\end{equation}

Let us choose basis $(e_0,e_1,e'_2)$ in $\mathfrak{g}_0,$ orthonormal relative to $ds_0^2.$ 

On the ground of (\ref{g}), $e'_2=\alpha(e_0-e_2),$ $\alpha\in\mathbb{R},$ and 
$$-1=(e'_2,e'_2)=\alpha^2\left(1-2+\frac{1}{2}\right)=
-\frac{\alpha^2}{2},\quad \alpha=\sqrt{2},\quad e'_2=\sqrt{2}(e_0-e_2).$$
Moreover $[e_1,e'_2]=\sqrt{2}e_2=\sqrt{2}(e_0-(e_0-e_2))=\sqrt{2}e_0-e'_2$
and in chosen orthonormal basis of the Lie algebra $\mathfrak{g}_0,$ all non-zero structure constant are the following
\begin{equation}
\label{str}
C^0_{12}=-C^0_{21}=\sqrt{2}, \quad C^2_{12}=-C^2_{21}=-1. 
\end{equation} 

We proved in theorems 2 and 3 from \cite{BerZub2} that every timelike or isotropic geodesic $\gamma=\gamma(t),$ $t\in \mathbb{R},$ on arbitrary Lie group $G$ (with origin at its unit) with left-invariant (sub-)Lorentz metric  $(\cdot,\cdot)$ and orthonormal relative to it basis $(e_0,e_1,\cdots,e_n)$ in the Lie algebra $\mathfrak{g}$ is a solution of the system of ordinary differential equations
\begin{equation}
\label{equat}
\gamma'(t)=dl_{\gamma(t)}(u(t)),\quad u(t)=\psi_0(t)e_0-\sum\limits_{i=1}^{r}\psi_i(t)e_i,\quad \psi_0(t)>0,
\end{equation}
\begin{equation}
\label{equat1}
\psi_j'(t)=\sum\limits_{k=0}^{n}\left(C_{0j}^{k}\psi_0\psi_k-\sum\limits_{i=1}^{r}C_{ij}^{k}\psi_i\psi_k\right),\,\,j=0,\dots,n.
\end{equation}
Here $C_{ij}^{k}$ are structure constants in the basis $(e_0,\dots, e_n)$ of the Lie algebra $\mathfrak{g}$. 

\begin{remark}
\label{opt}
Set $\psi(t)=(\psi_0(t),\psi_1(t),\dots, \psi_n(t))\in\mathbb{R}^{n+1},$	$$v=v_0e_0+\sum\limits_{i=1}^nv_ie_i\in\mathfrak{g},\quad	\psi(t)(v):=\sum\limits_{j=0}^n \psi_j(t)v_j,\quad U=\{u\in\mathfrak{g}: u_0>0, (u,u)\geq 1\}.$$
Then the equation (\ref{equat1}) and the equality for $u(t)\in \mathfrak{g}$ in (\ref{equat}) for the case of timelike geodesic are equivalent to conditions 
\begin{equation}
\label{equat2}	
\psi(t)(u(t))=(u(t),u(t))=\min_{u\in U}\psi(t)(u)\equiv 1, 	
\end{equation}
\begin{equation}
\label{equat3}	
\psi'(t)(v)= [u(t),v]=\ad(u(t))(v)\quad\mbox{for every}\quad v\in \mathfrak{g}. 	
\end{equation}
\end{remark}

\begin{definition}
\label{def}	
Closed non-compact strictly convex set $U$ is called the {\it control region}, $u(t),$ $t\in\mathbb{R}$ is {\it control function} (or simply {\it control}). Relations (\ref{equat2}), (\ref{equat3}) together with differential equation in  (\ref{equat}) is the left-invariant Pontryagin Minimum Principle for the slow-optimal problem.
\end{definition}  

Remark \ref{opt} and Definition \ref{def} are equivalent to (consisting of three principles) Theorem 12 from \cite{BerGich2001}, in which are stated necessary conditions for {\it directed into future longest timelike curves} of left-invariant metrics (more general than (sub-)Lorentz) on Lie groups $G$. Curves, satisfying such conditions, are called {\it extremals}. 

In general case, timelike geodesics of left-invariant (sub-)Lorentz metric are not globally, but only locally longest curves.

In the statements of Theorem 12 from \cite{BerGich2001}, unlike Remark \ref{opt}, is applied the adjoint representation $\Ad$ of the Lie group $G$. Let us notice that  
$$\ad(w)=\frac{d\Ad(\exp(sw))}{ds}(0),\quad w\in\mathfrak{g}.$$ 
Here $\exp$ is the exponential mapping of the Lie algebra $\mathfrak{g}$ 
into Lie group $G,$ $\exp(sw),$ $s\in\mathbb{R},$ is the one-parameter subgroup in $G$ with initial tangent vector $w,$ $\ad(\cdot)$ is the adjoint representation of the Lie algebra $(\mathfrak{g},[\cdot,\cdot]).$
Apparently such linearization of the representation $\Ad$ is useful since it reduces a part of the search problem for geodesics (although splits its solution by two steps) to algebraic problem, while differential equations (\ref{equat1}) are one-type for all Lie algebras. But for the matrix Lie groups a realization of $\Ad$ is also algebraic.

Notice that in case of Lie groups with left-invariant Lorentz metric Theorem  12 in \cite{BerGich2001} gives the same differential equations {\it of the first order} for timelike geodesics as \cite{Gav}; Theorem 12 is applicable for left-invariant sub-Lorentz metrics unlike  \cite{Gav}.  

We have $r=n=2; i,j,k=0,1,2.$ By (\ref{str}), equations (\ref{equat1}) take a form 
$$\psi'_j(t)=-\sum\limits_{k=0}^2\left(\sum\limits_{i=1}^2C^k_{ij}\psi_i(t)\psi_k(t)\right)=-
\left(\sum\limits_{i=1}^2C^0_{ij}\psi_i(t)\psi_0(t)+\sum\limits_{i=1}^2C^2_{ij}\psi_i(t)\psi_2(t)\right).$$

Using (\ref{str}) again, we obtain $\psi'_0(t)=0,\psi_0(t)\equiv \varphi_0,$
\begin{equation}
\label{dif}
\psi'_1(t)=\psi_2(t)(\sqrt{2}\varphi_0-\psi_2(t)),\quad \psi'_2(t)=-\psi_1(t)(\sqrt{2}\varphi_0-\psi_2(t)).
\end{equation}

\begin{remark}
In case of isotropic (timelike) geodesic, we assume that  
$(u(0),u(0))=0, \varphi_0\equiv \psi_0(t)\equiv 1$ ( $(u(0),u(0))=1, \varphi_0\geq 1$ ).
\end{remark}

Let us begin with timelike geodesic in case when $\varphi_0= 1$. Then  $$(u(t),u(t))\equiv (u(0),u(0))=1, \psi_1(t)=\psi_2(t)\equiv 0, u(t)\equiv u(0)= e_0,$$
\begin{equation}
\label{wl}	
\gamma(t)=\exp(te_0)= te_0=(x_0(t)= t,x_1(t)=0,x_2(t)=0, x_3(t)=0), t\in\mathbb{R}.
\end{equation}
Later in case of timelike geodesic we assume that  
$\varphi_0>1.$ 

\begin{proposition}
\label{ut}
Any geodesic in $(S_0,ds_0^2)$ satisfies equalities $$u(t)=(\varphi_0-b\sqrt{2}\cos\theta(t))e_0+b(\sin\theta(t)e_1
+\sqrt{2}\cos\theta(t)e_2),$$
where $\theta(t)$ is a solution of the ordinary differential equation 
\begin{equation}
\label{thet}
\frac{d\theta}{dt}=b\cos\theta(t)-a,
\end{equation}
where $a=\sqrt{2}\varphi_0,$ and $b=1,$ $b=\sqrt{\varphi^2_0-1},$ 
for isotropic and timelike geodesics respectively. 
\end{proposition}

\begin{proof}
We easily obtain from equalities 
$$(u(t),u(t))\equiv (u(0),u(0)),\quad \psi_0(t)\equiv \psi_0(0)=\varphi_0$$ and values $(u(0),u(0))$ for two mentioned kinds of geodesics that $\psi^2_1(t)+\psi^2_2(t)=b^2.$ 
Therefore we can suppose that  
\begin{equation}
\label{psi}
(\psi_1(t),\psi_2(t))=b(-\sin\theta(t),\cos\theta(t)).
\end{equation}
Then in consequence of (\ref{dif}),
$$\psi'_1(t)=b(-\cos\theta(t)\theta'(t))=b\cos\theta(t)(a-
b\cos\theta(t)),$$
$$\psi'_2(t)=b(-\sin\theta(t)\theta'(t))=b\sin\theta(t)(a-b\cos\theta(t)),$$ which implies (\ref{thet}). 	
\end{proof}		

\begin{corollary}
\label{der}	
$b>0,$ $\theta'(t) < 0$. 	
\end{corollary}

Separating variables in (\ref{thet}), we obtain according to formula 5.12.5 in \cite{BrMarPr86}
$$t-t_0=\int dt=\int \frac{d\theta}{-a+b\cos\theta}=\frac{2}{\sqrt{a^2-b^2}}\arctan\left(\frac{-(a+b)\tan(\theta/2)}{\sqrt{a^2-b^2}}\right)=$$
$$\frac{2}{\sqrt{a^2-b^2}}\arctan\left(-\sqrt{\frac{a+b}{a-b}}\tan(\theta/2)\right).$$ 
Hence we get consequently 
\begin{equation}
\label{the}
\tan\left(\frac{\theta}{2}\right)=-\sqrt{\alpha}\tan\sigma,\quad\mbox{where}\quad \alpha=\frac{a-b}{a+b},\quad \sigma=\frac{\sqrt{a^2-b^2}(t-t_0)}{2},	
\end{equation}
\begin{equation}
\label{cs}
\cos\theta(t)=\frac{1-\alpha\tan^2\sigma}{1+\alpha\tan^2\sigma},\quad \sin\theta(t)=\frac{-2\sqrt{\alpha}\tan\sigma}{1+\alpha\tan^2\sigma}.
\end{equation}

Below we shall need formulae (\ref{cs}) in the form
\begin{equation}
\label{cs1}
\cos\theta(t)=\frac{\cos^2\sigma-\alpha\sin^2\sigma}{\cos^2\sigma+\alpha\sin^2\sigma}=\frac{(1+\alpha)\cos 2\sigma+(1-\alpha)}{(1-\alpha)\cos 2\sigma+(1+\alpha)},
\end{equation}
\begin{equation}
\label{ss1} 
\sin\theta(t)=\frac{-2\sqrt{\alpha}\cos\sigma\cdot \sin\sigma}{\cos^2\sigma+\alpha\sin^2\sigma}=
\frac{-2\sqrt{\alpha}\sin 2\sigma}{(1-\alpha)\cos 2\sigma+(1+\alpha)}.
\end{equation}

For $t=2\sigma$ we find with help of 5.12.3 and 5.12.5 from \cite{BrMarPr86} indefinite integrals of (\ref{cs1}) and (\ref{ss1}), omitting additive constants:
$$\int\frac{(1-\alpha)+(1+\alpha)\cos t}{(1+\alpha)+(1-\alpha)\cos t}dt=
\frac{1}{1-\alpha}\left[(1+\alpha)t-4\sqrt{\alpha}\arctan(\sqrt{\alpha}\tan(t/2))\right],$$
$$\int\frac{-2\sqrt{\alpha}\sin t}{(1+\alpha)+(1-\alpha)\cos t}dt=\frac{2\sqrt{\alpha}}{1-\alpha}\ln((1+\alpha)+(1-\alpha)\cos t).$$ 
With the help of 5.12.1 from \cite{BrMarPr86} we compute the following integral 
$$\int\frac{(1-\alpha)+(1+\alpha)\cos t}{[(1+\alpha)+(1-\alpha)\cos t]^2}dt=\frac{\sin t}{(1+\alpha)+(1-\alpha)\cos t}.$$

\begin{remark}
\label{as}	
In case of timelike and isotropic geodesics respectively 
$$\alpha=\alpha_1=\frac{\sqrt{2}-\sqrt{1-(1/\varphi_0)^2}}{\sqrt{2}+\sqrt{1-(1/\varphi_0)^2}},\quad \sigma=\sigma_1=\frac{\sqrt{1+\varphi^2_0}(t-t_0)}{2};$$
$$\alpha=\alpha_2=\frac{\sqrt{2}-1}{\sqrt{2}+1}=(3+2\sqrt{2})^{-1},\quad \sigma=\sigma_2=\frac{t-t_0}{2}.$$
In addition, $\alpha_1>\alpha_2,$ $|\sigma_1|>\sqrt{2}|\sigma_2|;$ inequalities in both cases are not improvable. 
\end{remark}

We shall solve (the first) matrix differential equation in (\ref{equat}) using Proposition \ref{ut}, equalities (\ref{e}), (\ref{cs1}), (\ref{ss1}),
and Remark \ref{as} and assuming that $(4\times 4)$-matrix $\gamma(t)$ looks as in (\ref{B}) with unknown functions $x_0(t),$ $x_1(t),$ $x_2(t),$ $t\in\mathbb{R}.$ Moreover in case of matrix Lie group $G_0,$ the operation of the differential $dl_{\gamma(t)}$ of left shift $l_{\gamma(t)}$ by element $\gamma(t)$ in (\ref{equat}) is realized as multiplication by the matrix $\gamma(t)$ from the left. 

\begin{theorem}
\label{isotr}
Isotropic geodesics in $G_0$ starting at its unit have the following form
\begin{equation}
\label{x0}	
x_0(t)=[-\tau+2\sqrt{2}\arctan((\sqrt{2}-1)\tan(\tau/2)]|^{t-t_0}_{-t_0},
\end{equation}
\begin{equation}
\label{x1}
x_1(t)= \ln((1+\alpha_2)+(1-\alpha_2)\cos\tau)|^{t-t_0}_{-t_0},
\end{equation}
\begin{equation}
\label{x2}
x_2(t)=\frac{\beta_2\sin \tau}{(1+\alpha_2)+(1-\alpha_2)\cos\tau}|_{-t_0}^{t-t_0}, 
\end{equation}
$$\beta_2=\sqrt{2}[(1+\alpha_2)+(1-\alpha_2)\cos t_0].$$

In case $\varphi_1=\psi_1(0)=0,$ $\sin(0)=0$ and for $t_0=0,$  
\begin{equation}
\label{x02}
x_0(t)=-t+ 2\sqrt{2}\arctan((\sqrt{2}-1)\tan(t/2)), 
\end{equation}
\begin{equation}
\label{x12}
x_1(t)=\ln[((1+\alpha_2)+(1-\alpha_2)\cos t)/2],
\end{equation}
\begin{equation}
\label{x22}
x_2(t)=\frac{2\sqrt{2}\sin t}{(1+\alpha_2)+(1-\alpha_2)\cos t}.	
\end{equation}
\end{theorem}

\begin{proof}
We obtain the following differential equations for elements of the matrix  $\gamma(t)$ with non-zero right parts
$$\gamma'_{11}(t)=\frac{-2\sqrt{\alpha_2}e^{-x_1}\sin(t-t_0)}
{(1-\alpha_2)\cos (t-t_0)+(1+\alpha_2)},$$ 
$$\gamma'_{14}(t)=x'_2(t)=e^{-x_1}\sqrt{2}\frac{(1+\alpha_2)\cos (t-t_1)+(1-\alpha_2)}{(1-\alpha_2)\cos(t-t_0)+(1+\alpha_2)},$$
$$\gamma'_{24}(t)=x'_1(t)=\frac{-2\sqrt{\alpha_2}\sin (t-t_0)}{(1-\alpha_2)\cos (t-t_0)+(1+\alpha_2)},$$
\begin{equation}
\label{x000}	 
\gamma'_{34}(t)=x'_0(t)=1-\sqrt{2}\frac{(1+\alpha_2)\cos (t-t_0)+(1-\alpha_2)}{(1-\alpha_2)\cos (t-t_0)+(1+\alpha_2)}.
\end{equation}
It is enough to solve the last three differential equations. 

Applying the indefinite integrals which we found earlier for  $\alpha=\alpha_2$, we get
\begin{equation}
\label{x03} 
x_0(t)=t - \frac{\sqrt{2}(1+\alpha_2)t}{1-\alpha_2} + \frac{4\sqrt{2}}{1-\alpha_2}\sqrt{\alpha_2}
\arctan(\sqrt{\alpha_2}\tan\tau)|^{(t-t_0)/2}_{-t_0/2},
\end{equation}
\begin{equation}
\label{x13}
x_1(t)=\frac{2\sqrt{\alpha_2}}{1-\alpha_2}\ln((1+\alpha_2)+(1-\alpha_2)
\cos\tau)
|^{t-t_0}_{-t_0},
\end{equation}
whence follow (\ref{x0}), (\ref{x1}). Now, using  (\ref{x1})
and the last integral, computed before Theorem \ref{isotr}, we obtain  (\ref{x2}).
\end{proof}

\begin{theorem}
\label{teor}
Timelike geodesics in $G_0$ starting at its unit have the following view

$$x_0(t)=-\varphi_0(t-t_0) + 2\sqrt{2}\arctan\left(\sqrt{\alpha_1}\tan\left(\frac{\sqrt{\varphi_0^2+1}
	(t-t_0)}{2}\right)\right)$$

$$-\left[\varphi_0t_0+2\sqrt{2}\arctan\left(\sqrt{\alpha_1}\tan
\left(\frac{-\sqrt{\varphi_0^2+1}t_0)}{2}\right)\right)\right],$$

$$
x_1(t)=\ln((1+\alpha_1)+(1-\alpha_1)\cos\tau)
|^{\sqrt{\varphi^2_0+1}(t-t_0)}_{-\sqrt{\varphi^2_0+1}t_0},$$

$$x_2(t)=\frac{\beta_1\sin \tau}{(1+\alpha_1)+(1-\alpha_1)\cos\tau}|_{-\sqrt{\varphi_0^2+1}t_0}^{\sqrt{\varphi_0^2+1}(t-t_0)},$$ 
$$\beta_1=\frac{\sqrt{2(\varphi^2_0-1)}}{\sqrt{\varphi^2_0+1}}
\left[(1+\alpha_1)+(1-\alpha_1)\cos\sqrt{\varphi_0^2+1}t_0\right]. $$

In case $\varphi_1=\psi_1(0)=0,$ $\sin(0)=0$ and for $t_0=0,$ 
$$x_1(t)=\ln\left[\left((1+\alpha_1)+(1-\alpha_1)\cos\left(\sqrt{\varphi_0^2+1}t\right)\right)/2\right],$$

$$x_0(t)=-\varphi_0t + 2\sqrt{2}\arctan\left(\sqrt{\alpha_1}\tan\left(\sqrt{\varphi_0^2+1}t/2\right)\right),$$
$$x_2(t)=\frac{2\sqrt{2(\varphi^2_0-1)}}{\sqrt{\varphi^2_0+1}}\left[\frac{\sin(\sqrt{\varphi_0^2+1}t)}{(1+\alpha_1)+(1-\alpha_1)\cos(\sqrt{\varphi_0^2+1}t)}\right].$$	
\end{theorem}

\begin{proof}
We have the following differential equations 
$$\gamma'_{34}(t)=x'_0(t)=\varphi_0-\sqrt{2(\varphi_0^2-1)}
\frac{(1+\alpha_1)\cos(\sqrt{\varphi_0^2+1}(t-t_0))+(1-\alpha_1)}
{(1-\alpha_1)\cos(\sqrt{\varphi_0^2+1}(t-t_0))+(1+\alpha_1)}$$.
$$\gamma'_{24}(t)=x'_1(t)=\frac{-2\sqrt{\alpha_1(\varphi_0^2-1)}\sin (\sqrt{\varphi_0^2+1}(t-t_0))}{(1-\alpha_1)\cos (\sqrt{\varphi_0^2+1}(t-t_0))+(1+\alpha_1)},$$
$$\gamma'_{14}(t)=x'_2(t)=e^{-x_1}\sqrt{2(\varphi_0^2-1)}\frac{(1+\alpha_1)
\cos(\sqrt{\varphi_0^2+1}(t-t_0))+(1-\alpha_1)}{(1-\alpha_1)\cos(\sqrt{\varphi_0^2+1}(t-t_0))+(1+\alpha_1)}.$$

Now Theorem \ref{teor} is proved by using arguments and computations analogous to those given in the proof of Theorem \ref{isotr}.

\end{proof}

\begin{proposition}
\label{arc}	
Function $f(t)=\arctan(\beta\tan t),$ $t\in\mathbb{R},$ where $\beta>0,$ has a bounded positive derivative on the set of all numbers $t\in\mathbb{R}$ such that $t\neq \pi/2+k\pi,$ $k\in\mathbb{Z};$ $f$ is continuous and strictly increases.
\end{proposition}

\begin{proof}
Let $t\neq \pi/2+k\pi$. Then
$$f'(t)=\frac{\beta(1+\tan^2t)}{1+(\beta\tan t)^2}\rightarrow \frac{1}{\beta},\quad\mbox{if}\quad \tan t\rightarrow \pm\infty.$$
From here follow all statements of Proposition \ref{arc}.
\end{proof}

Proposition \ref{arc}, with addition of not large arguments, implies	

\begin{corollary}
If $\beta>0,$ then 
\begin{equation}
\label{arct}
f(\beta\tan(t+k\pi)) = \arctan(\beta\tan t) + k\pi;\quad -\frac{\pi}{2}\leq t \leq \frac{\pi}{2},\quad k\in\mathbb{Z}.	
\end{equation}		
\end{corollary}

\begin{theorem}
\label{maint}
G\"odel Universe has neither isotropic nor timelike closed geodesics. 	
\end{theorem}

\begin{proof}
Assume that there exists a closed isotropic or timelike geodesic  $\gamma=\gamma(t),$ $t\in \mathbb{R},$ in $(S,ds^2),$ in other words, on the Lie group $(G,ds^2)$. Since $ds^2$ is left-invariant, we can suppose that  $\gamma$ starts at the unit of Lie group $G,$ i.e. 
$x_0(0)=x_1(0)=x_2(0)=x_3(0)=0.$ Since the vector $e_3\in\mathfrak{g}$ commutes with all vectors from Lie algebra $\mathfrak{g},$ then in consequence of differential equations (\ref{equat1}) for $\gamma$, $\psi'_3(t)\equiv 0,$ i.e. $\psi_3(t)\equiv \varphi_3=\psi_3(0).$ By (\ref{equat}), $x_3(t)=\varphi_3\cdot t.$ This functions can take the same values for different values of $t$ only if $\varphi_3=0.$ Then 
$x_3(t)\equiv 0,$ i.e. $\gamma\in S_1,$ and $\gamma$ is a closed geodesic on the Lie group $(G_0,ds_0^2)$ starting at its unit.

It folows from Theorem \ref{isotr} (respectively, Theorem \ref{teor}) that  $x_1(t)$ and $x_2(t)$ have the least common positive period $2\pi$ 
(respectively, $2\pi/\sqrt{\varphi_0^2+1}).$ Therefore $x_1(t_2)-x_1(t_1)=0$
and $x_2(t_2)-x_2(t_1)=0$ simultaneously if and only if 
$t_2-t_1=2\pi k$ (respectively, $t_2-t_1=2\pi k/\sqrt{\varphi_0^2+1}.$)

If $k\neq 0,$ then in consequence of Proposition \ref{arc} and Theorem \ref{isotr} (respectively, Theorem \ref{teor}) 
$x_0(t_2)-x_0(t_1)=2\pi k(\sqrt{2}-1)\neq 0$ (respectively, $x_0(t_2)-x_0(t_1)=2\pi k(\sqrt{2}-\varphi_0/\sqrt{\varphi_0^2+1})\neq 0$).

Theorem \ref{maint} is proved.  	
\end{proof}

Notice that it is possible to prove Theorem \ref{maint} for the case of isotropic geodesics, if we shall use instead of (\ref{arct}) the following
proposition.

\begin{proposition}
\label{spint}
For every isotropic geodesic and arbitrary $t\in\mathbb{R},$ 
$$x_0(t+2\pi)= x_0(t)+ T,$$
where $T$ is some positive constant. 
\end{proposition}

\begin{proof}
For simplicity, assume that $t_0=0.$ Then on the ground of (\ref{x000}), 
$$x_0'(t)=1-\sqrt{2}\frac{(1+\alpha_2)\cos t+(1-\alpha_2)}{(1-\alpha_2)\cos t+(1+\alpha_2)}=$$
$$\frac{(1-\alpha_2)\cos t+(1+\alpha_2)-\sqrt{2}[(1+\alpha_2)\cos t+(1-\alpha_2)]}{(1-\alpha_2)\cos t+(1+\alpha_2)}=$$
$$\frac{(1-\sqrt{2})+\alpha_2(1+\sqrt{2})+\cos t[(1-\sqrt{2})-\alpha_2(1+\sqrt{2})]}{(1-\alpha_2)\cos t+(1+\alpha_2)}=$$
$$\frac{-\sqrt{\alpha_2}+\sqrt{\alpha_2}\cdot 1 +\cos t[-\sqrt{\alpha_2}-\sqrt{\alpha_2}\cdot 1]}{(1-\alpha_2)\cos t+(1+\alpha_2)}=
\frac{-2\sqrt{\alpha_2}\cos t}{(1-\alpha_2)\cos t+(1+\alpha_2)}.$$
Now, using the periodicity of the function $\cos t,$
$$x_0(t+2\pi)-x_0(t)=-2\sqrt{\alpha_2}\int_t^{t+2\pi}\frac{\cos\tau d\tau}{(1-\alpha_2)\cos\tau+(1+\alpha_2)}=$$
$$-2\sqrt{\alpha_2}\int_{-\pi/2}^{3\pi/2}\frac{\cos t dt}{(1-\alpha_2)\cos t+(1+\alpha_2)}=$$
$$-2\sqrt{\alpha_2}\int_{-\pi/2}^{\pi/2}\left[\frac{\cos t}{(1-\alpha_2)\cos t+(1+\alpha_2)}+\frac{\cos(t+\pi)}{(1-\alpha_2)
\cos(t+\pi)+(1+\alpha_2)}\right]dt=$$
$$2\sqrt{\alpha_2}\int_{-\pi/2}^{\pi/2}\left[\frac{-\cos t}{(1-\alpha_2)\cos t+(1+\alpha_2)}+\frac{\cos t}
{-(1-\alpha_2)\cos t+(1+\alpha_2)}\right]dt:= T>0.$$
\end{proof}

\begin{remark}
The proof of Proposition \ref{spint} and (\ref{x0}),  (\ref{arct}) (for $\beta=\sqrt{\alpha_2}$) imply that 
$$T=2(\sqrt{2}-1)\pi,\quad \int_0^{2\pi}\frac{\cos t dt}{(1+\alpha_2)+(1-\alpha_2)\cos t}=-\pi.$$	
\end{remark}

One can easily prove the following theorems.

\begin{theorem}
\label{gisotr}		
An isotropic geodesic in $(G,ds^2),$ starting at the unit of $G$ and not lying in the subgroup $G_0,$ can have one of two forms:
\begin{equation}
\label{spec}	
x_0(t)=x_3(t)=t,\quad x_1(t)=x_2(t)\equiv 0,\quad t\in\mathbb{R},	 
\end{equation}
\begin{equation}
\gamma(t)=(\tilde{\gamma}(t),\varphi_3t),\quad 0< |\varphi_3|<1, \quad t\in\mathbb{R},	
\end{equation}
where $\tilde{\gamma}(t)$ is obtained from formulae in Theorem \ref{teor} by replacements 
$$-\varphi_0t\rightarrow -t,\alpha_1\rightarrow \frac{\sqrt{2}-\sqrt{1-\varphi^2_3}}{\sqrt{2}+\sqrt{1-\varphi^2_3}},
\sqrt{\varphi_0^2+1}\rightarrow \sqrt{\varphi_3^2+1}, \frac{\sqrt{2(\varphi^2_0-1)}}{\sqrt{\varphi^2_0+1}}\rightarrow \frac{\sqrt{2(1-\varphi_3^2)}}{\sqrt{\varphi^2_3+1}}.$$
 \end{theorem}

\begin{theorem}
\label{gtl}
A timelike geodesic in $(G,ds^2),$ starting at the unit of $G$ and not lying in the subgroup $G_0,$ can have one of two forms:		
\begin{equation}
\label{spec1}	
x_0(t)=\varphi_0t,\quad x_3(t)=\sqrt{\varphi_0^2-1}t,\quad x_1(t)=x_2(t)\equiv 0,\quad t\in\mathbb{R},\quad \varphi_0>1, 
\end{equation}
\begin{equation}
\gamma(t)=(\tilde{\gamma}(t),\varphi_3t),\quad 0< |\varphi_3|< \sqrt{\varphi_0^2-1}, \quad t\in\mathbb{R},	
\end{equation}
where $\tilde{\gamma}(t)$ is obtained from formalae in Theorem \ref{teor} by replacements
$$\alpha_1\rightarrow \frac{\sqrt{2}-\sqrt{1-(1+\varphi^2_3)/\varphi_0^2}}{\sqrt{2}+
	\sqrt{1-(1+\varphi^2_3)/\varphi_0^2}},\quad
\sqrt{\varphi_0^2+1}\rightarrow \sqrt{\varphi_0^2+1+\varphi^2_3},$$ $$\frac{\sqrt{2(\varphi^2_0-1)}}{\sqrt{\varphi^2_0+1}}\rightarrow \frac{\sqrt{2(\varphi_0^2-(1+\varphi^2_3))}}{\sqrt{\varphi_0^2+1+\varphi^2_3}}.$$
\end{theorem}

\section{Concluding remarks}
\label{last}

\begin{remark}
The absence of closed timelike geodesics in G\"odel Universe can be deduced
conceptually from the absence of closed isotropic geodesic, without application of Theorem \ref{teor}.  
\end{remark}

\begin{theorem}
\label{ch}
If in the first equality from formula (35) in \cite{ChWr} and in formulae 
(26), (30), (33) from \cite{ChWr} to set $D=\sqrt{2},$ $B=1,$ 
(resp., $D=\sqrt{2}\varphi_0,$ $B=\sqrt{\varphi_0^2-1}$), 
and to take $c_0$, $c_1,$ $c_2$ in such a way that the corresponding curve starts at the unit from $G_0,$ then one obtains the geodesics from Theorem 
\ref{isotr} (resp., Theorem \ref{teor}).  	
\end{theorem}

\begin{proof}
The function $x_1(t)$ from Theorem \ref{isotr} (resp., Theorem \ref{teor}) 
will not change if we divide the argument in $\ln (\cdot)$ by 2. Then the
application of indicated values of $D, B$, formulae from \cite{ChWr}, and known trigonometric relations for the transiti\-ons from $2\sigma$ to $\sigma$ and back completes the proof of Theorem \ref{ch}.
\end{proof}	

\begin{remark}
In \cite{ChWr}, the authors admit the cases $B=0$ (world lines of matter) and $B< 0.$ Then, according to (37) in \cite{ChWr}, $(\sqrt{2}-1)^2\leq \alpha \leq (\sqrt{2}+1)^2.$	
\end{remark}	

The space $S_0$ is obtained from the space $R$ with constant positive curvature and signature $(+,-,-)$ by expanding in ratio $\sqrt{2}:1$ in direction of the system $P$ of oriented timelike Clifford parallels. This definition of $S_0$ also leads to an elegant presentation of its isometry group \cite{God1949}. Let us explain that the system $P$ consists of world lines of matter in $S_0$ and the group $G_0$ is the subgroup in the isometry group $G_R$ of the space $R$, transforming the system $P$ into itself. We need to omit here the given in  \cite{God1949} description of Clifford parallels, the space $R$ and the group  $G_R$ by means of hyperbolic quaternions and fundamental Klein quadric.

The mentioned expanding coefficient $\mu=\sqrt{2}$ is chosen in order to get the most simple form of the Ricci tensor indicated above and in \cite{God1949}.

Section 5.7 in \cite{HockEll} is called ''G\"odel's universe''.
The orbits (circles) of one-parameter group of isometric rotations in $S_0$ around the axis $r=0$ for points of the plane $t=0$ in coordinates $(r,\varphi,t)$ are depicted on Figure 31.  For $0< r < \ln (\sqrt{2}+1)$ they are spacelike, for $r=\ln (\sqrt{2}+1)$ are isotropic, for $r> \ln (\sqrt{2}+1)$ are timelike. Let us note that $\sqrt{2}+1=1/\sqrt{\alpha_2}.$

At the end of section 5.7 is given also the verbal description of timelike and isotropic geodesics. Let us repeat literally the last description:
''The null geodesics from a point $p$ on the axis of coordinates (Figure 31) diverge from the axis initially, reach a caustic at $r=\ln (\sqrt{2}+1)$, and then reconverge to a point $p'$ on the axis''. That is,
{\it there is no closed isotropic geodesics}. Moreover, there is no 
reference in \cite{HockEll} to a source of this statement and geodesics really are not studied in section 5.7. In the last sentence of this section
is given only references to \cite{God1949}, \cite{Kundt}. There is no reference to \cite{ChWr} in \cite{HockEll}. In \cite{HockEll} is said that ''the solution is not very physical''.

An isotropic geodesic of $S_0$ is presented in \cite{ChWr} as a helix on the cylinder $r=\ln(\sqrt{2}+1)/2$; this helix is an orbit of some
one-parameter isometry group of $S_0$.

Ellis gives in \cite{Ellis} a deep analysis of articles \cite{God1949}, \cite{God1949a}, \cite{God1950} and their influence onto later investigations in General Relativity Theory. Let us give some citations on  \cite{God1949}:

''Curiously, the beginning of the modern studies of singularities in general relativity in many ways had its seeds in the presentation by Kurt G\"odel  (1949) of an exact solution of Einstein's equations for pressure-free matter, which could be thought of as a singularity-free, rotating but non-expanding cosmological model'' (p. 34).

''The electric part of the Weyl tensor (for G\"odel Universe) is non-zero and is given by (2.3) but the magnetic part of the Weyl tensor is zero. Because of rotational symmetry, the Weyl tensor is Petrov type $D$'' (p. 36). (This is consistent with \cite{Steph}).

''G\"odel did not describe the geodesic properties of this space-time
but may have investigated them. Later investigations by Kundt \cite{Kundt}, Chandrasekhar and Wright \cite{ChWr} explicitely showed that there is no closed timelike geodesics'' (p. 37). 

Note that in \cite{Ellis} there is nothing about connection of these works with isotropic geodesics but on p. 38 with reference to \cite{HockEll} is mentioned on these geodesics.

In abstract to \cite{Kundt}, the author (then a student of Pascual Jordan) writes that in his paper are integrated differential equations of geodesics in some inertial system resting relative to a world line of matter.

We could not deduce on the base of formulas in \cite{Kundt} the judgement that in \cite{Kundt} was proved the absence of closed timelike geodesics.

Kundt in paper \cite{Kundt} accepts (\ref{lig1}) as initial linear element for $a=1$ and obtains after transformation of coordinates  
$$x_1=-\ln y,\quad x_2=\sqrt{2}x,\quad 0< y< \infty,\quad -\infty < x <\infty,\quad x_0=t,\quad x_3=z:$$
\begin{equation}
\label{lob}
ds^2=\left(dt+\frac{\sqrt{2}dx}{y}\right)^2-\left(\frac{dx^2+dy^2}{y^2}
+dz^2\right).
\end{equation}
One can easily check that both the temporal and the space components of this metric are invariant relative to action of the Lie group $G$.  

As a corollary of (\ref{lob}), the linear element $ds^2_0$ in (\ref{lin}) can be written in the form
\begin{equation}
\label{lig2}
ds_0^2=\left(dt+\frac{\sqrt{2}dx}{y}\right)^2-
\left(\frac{dx^2+dy^2}{y^2}\right).	  
\end{equation}

It is known that the space components of the metric (\ref{lig2}) is the Lobachevsky plane with Gaussian curvature $-1$ in the Poincare model at
the upper half-plane. 

Once more, both the temporal and the space components of this metric are invariant relative to action of the Lie group $G_0$.

\begin{remark}
More precisely, on the space component of (\ref{lig2}) acts simply transitively by isometries the Lie group $G_2.$ It is worthwhile to note that $G_2$ is isomorphic to the group of preserving orientation affine transformations of the straight line $\mathbb{R}.$ Natural generalizations of Lie group Ли $G_2$ are connected $(n+1)$-dimensional Lie groups  $G_{n+1},$ acting on Euclidean spaces $\mathbb{R}^n,$ $n\geq 1,$ by parallel translations and proper similarities. Sometimes they are called the basic affine groups. J. Milnor proved in \cite{Miln} that every left-invariant Riemannian metric on $G_{n+1}$ has constant negative sectional curvature, that is presents $(n+1)$-dimensional Lobachevsky space. It is connected with the fact that any connected $(n+1)$-dimensional Lie group,  locally isomorphic to $G_{n+1},$ is simply connected. Groups $G_{n+1},$ $n\geq 1,$ are all connected non-commutative Lie groups admitting no left-invariant sub-Riemannian metrics.
\end{remark}

A.V.~Levichev investigated in \cite{Lev1990} conditions of timelike and isotropic completeness, timelike convergence,  and weak energetic condition for all left-invariant Lorentz metrics on the group $(\mathbb{R},+)\times G_2\times (\mathbb{R},+)$. For all metrics are composed systems of differential equations of the first order for geodesics starting at the unit. Results are proved in Theorem 1 and summed up in Table from  \cite{Lev1990}. Table shows that only metrics of system 2b), which includes the metric (\ref{lig}), satisfy all three conditions.

Also Levichev studied the causal structure of those considered Lorentz manifolds which are geodesic complete. In Theorem 2, with reference to Theorem 2 from \cite{Lev19900}, it is established that the Lorentz manifolds from the system 2b) are {\it totally vicious}, that is any two points of the manifold can be connected by a timelike curve. In \cite{Lev1992} are also presented the other such homogeneous Lorentz manifold.

This is a total contrast to behavior of geodesics in G\"odel Universe. On the base of Theorems \ref{isotr}, \ref{teor}, it is not difficult to obtain the following (rough) estimates.

\begin{proposition}
All isotropic and timelike geodesics $\gamma(t)=(x_0(t),x_1(t),x_2(t))$ from
Theorem \ref{isotr} and \ref{teor} are contained in the cylinder $$\mathbb{R}\times D,\quad\mbox{где}\quad D=[-1,03; 0,7]\times [-(2+\sqrt{2}), 2+\sqrt{2}].$$ 
All isotropic and timelike geodesics of the form  $$\gamma(t)=(x_0(t),x_3(t),x_1(t),x_2(t))$$ in $(G,ds^2),$ starting at the unit, are contained in the set $F\times D$, where  
$$F=\{(x_0,x_3)\in \mathbb{R}^2: -x_0\leq x_3\leq x_0,\quad  x_0\geq 0\}.$$   		
\end{proposition}

Let us quote parts of abstracts to papers \cite{Andr17}, \cite{Andr21} on geodesics in G\"odel Universe, including spacelike. From the first abstract: ''The system of equations, which determines geodesic lines, is solved in elementary functions''. From the second one: ''The first four integrals of the geodesic lines system are given. The author clarifies the physical meaning of integration constants. Resulting from the complete integration of the geodesic lines system of equations in the Cartesian coordinates system, five families of geodesic lines are identified.'' 

In both papers in the search for geodesics is applied the classical method.
In References to papers are absent \cite{Kundt}, \cite{ChWr}, \cite{HockEll}. In References to the first paper is also absent \cite{God1949}. Analogs of Theorem \ref{ch} are not applicable to  \cite{Andr17}, \cite{Andr21}. In \cite{Andr17}, \cite{Andr21} there are no definite statements on existence or absence of closed timelike or isotropic geodesics for general cases.


\begin{thebibliography}{99}
	
\bibitem{God1949}
{\sl G\"odel~K.} An example of a new type of cosmological solutions of Einstein's field equations of gravitation.~ Rev. Mod. Phys., 21:3 (1949), 447--450.


\bibitem{God1949a}
{\sl G\"odel~K.} A remark about the relashionship between Relativity theory and idealistic Philosophy. In: A.P.~Schilpp (Ed.) Albert Einstein: Philosopher-Scientist. Volume VII in the Library of Living Philosophers. MJF Books, New York, 2001. P. 555--562.

\bibitem{Einst}
{\sl Einstein~A.} Reply to criticisms. In: A.P.~Schilpp (Ed.) Albert Einstein: Philosopher-Scientist. Volume VII in the Library of Living Philosophers. MJF Books, New York, 2001. P. 663--688.

\bibitem{God1950}
{\sl G\"odel~K.} Rotating universes in general relativity~ Proc. Int. Congr. Math. Cambr. Mass., 1 (1950), 175--181.

\bibitem{OzsSchuck}
{\sl Ozsv\'ath~I., Sch\"ucking~E.} Approaches to G\"odel rotating universe~ Class. Quantum Grav. 18(2001), 2243--2252.


\bibitem{GodColl}
{\sl G\"odel~K.} Collected Works. Volume III (Unpublished essay and lectures). Feferman~S., Dawson~J.W. Jr., Goldfarb~W., Parsons~Ch., Soloway~R.M. (Eds.) New York, Oxford: Oxford University Press, 1995.

\bibitem{Gam1946}
{\sl Gamov~G.} Rotating Universe?~ Nature, 158 (1946), 549--549.

\bibitem{Gam1952}
{\sl Gamov~G.} The role of turbulence in the evolution of the Universe.~ Phys. Rev. 86 (1952), 251--251.

\bibitem{ZeldNov}
{\sl Zel'dovich~Ya.B., Novikov~I.D.} The structure and evolution of the Universe. University of Chicago Press. Chicago, 1983.
 

\bibitem{Kundt}
{\sl Kundt~W.} Tr\"agheitsbahnen in einem von G\"odel angegebenen kosmologischen Modell. ~ Zeitschrift f\"ur Physik, 145:5 (1956), 611--620.


\bibitem{ChWr}
{\sl Chandrasekhar~S., Wright~J.P.} The geodesics in G\"odel universe. ~ Proc. Natl. Acad. Sci. USA, 47 (1961), 341--347.

\bibitem{Synge}
{\sl Synge~J.L.} Relativity: The General Theory. North Holland Publishing Company. Amsterdam, 1960.


\bibitem{Stein}
{\sl Stein~H.} On the paradoxical time-structures of G\"odel. ~ Philosophy of Science, 37:4 (1970), 589--601.

\bibitem{Guts1973}
{\sl Guts~A.K.} On the timelike closed smooth curves in general theory of Relativity. (Russian). ~ Russian Phys. (Iz. VUZ), 1973, no. 9, pp. 33--36.


\bibitem{HockEll}
{\sl Hawking~S.W., Ellis~G.F.R.} The Large Scale structure of Space-Time. 
Cambridge University Press. Cambridge, 1973.


\bibitem{RayTha}
{\sl Raychaudhuri~A.K., Thakurta~S.N.G.} Homogeneous spacetimes of the G\"odel type. ~ Phys. Rev. D., 22:4 (1980), 802--806.


\bibitem{Guts1980}
{\sl Guts~A.K.} The topological structure of the G\"odel Universe. 
(Russian). ~ Russian Phys. (Iz. VUZ), 1980, no.  6, pp. 109--110.


\bibitem{Lev19900}
{\sl Levichev~A.V.} Methods for studing the causal structure of homogeneous Lorentz manifolds. ~ Siberian Math. J., 31:3 (1990), 395--408.
DOI:https://doi.org/10.1007/BF00970346


\bibitem{Lev1990}
{\sl Levichev~A.V.} Causal structure of left-invariant Lorentz metrics on the Lie group $M_2\times \mathbb{R}^2$. ~ Siberian Math. J., 31:4 (1990), 607--614. DOI:https://doi.org/10.1007/BF00970631


\bibitem{Lev1992}
{\sl Levichev~A.V.} Three totally vicious homogeneous Lorentz spaces. (Russian). Proceedings of the Sobolev Institute of Mathematics, 21(1992), 122--131. https://www.mathnet.ru/rus/mt/v21/p122

\bibitem{Ellis}
{\sl Ellis~G.F.R.} Contribution of Kurt G\"odel to Relativity and Cosmology. in: Lecture Notes in Logic 6. Petr H\'ajek (Ed.) G\"odel 96. Logical Foundations of Mathematics, Computer Science and Physics --- Kurt G\"odel Legacy. Brno, Czech Republic. August 1996, Proceedings. Berlin, Springer-Verlag, Berlin, Heidelberg, 1996. P. 34--49.

\bibitem{Steph}
{\sl Stephani~H. et al.} Exact solutions of Einstein's field equations. 2nd ed. Cambridge Univ. Press, 2003.

\bibitem{ShikYush}
{\sl Shikin~G.N., Yushchenko~A.P.} The energy spectrum of charged scalar particles in the G\"odel Universe. (Russian). ~ Bulletin of Peoples' Friendship University of Russia. Series: Mathematics. Physics. 2011, no. 3. pp. 112--118.



\bibitem{Andr17}
{\sl Andronikova~E.O.} On geometric properties of the G\"odel universe. (Russian). ~ Bulletin of Moscow Region State University. Series: Physics and Mathematics, 2017, no. 1, pp. 51--56. DOI:10.18384/2310-7251-2017-1-51-56.


\bibitem{Andr21}
{\sl Andronikova~E.O.} On geometric properties of the pseudo-Riemannian manifold of the G\"odel Universe. (Russian). ~ RSUH/RGGU Bulletin. ''Information Science. Information security. Mathematics'' Series, 2021. no. 1, pp. 66--80. DOI:10.28995/2686-679X-2021-1-66-80.


\bibitem{BerGich2001}
{\sl Berestovskii~V.N., Gichev~V.M.} Metrized semigroups. ~ J. of Math. Sciences (United States). 119:1 (2004), 10--29 (20). 


\bibitem{Ber2023}
{\sl Berestovskii~V.N.}
To the Segal Chronometric Theory. ~ Siber. Adv. in Math. 33:3 (2023), 165-180. https://doi.org/10.1134/S105513442303001X


\bibitem{BerZub2}
{\sl Berestovskii~V.N., Zubareva~I.A.} Sub-Lorentzian geodesics on $GL^{+}(2,\mathbb{C})$ with the generating space of hermitian matrices in the Lie algebra $\mathfrak{gl}^{+}(2,\mathbb{C})$. ~ https://arxiv.org/abs/2310.08905

\bibitem{Gav}
{\sl Gavrilov~S.P.} Geodesics of left-invariant metrics on connected two-dimensional non-abelian Lie group. (Russian). ~ Gravitation and the Relativity Theory. Kazan, 1981. Issue 18. pp. 28--44.

\bibitem{BrMarPr86}
{\sl Brychkov~Yu.A., Marichev~O.I., Prudnikov~A.P.} Tables of indefinite integrals. (Russian). M.: Nauka, 1986.

\bibitem{Miln}	
{\sl Milnor~J.} Curvatures of left invariant metrics on Lie groups. ~ Adv. Math., 21:3 (1976), 293--329. 


\end{thebibliography}
\end{document}